\documentclass[12pt]{amsart}
\usepackage{graphicx}
\usepackage{a4wide}
\usepackage{amssymb}
\usepackage{amsthm}
\usepackage{amsmath}
\usepackage{amscd}
\usepackage{verbatim}
\usepackage[all]{xy}

\theoremstyle{plain}
\newtheorem{theorem}{Theorem}[section]
\newtheorem{corollary}[theorem]{Corollary}
\newtheorem{lemma}[theorem]{Lemma}
\newtheorem{proposition}[theorem]{Proposition}

\theoremstyle{definition}

\newtheorem{example}[theorem]{Example}
\newtheorem{remark}[theorem]{Remark}

\theoremstyle{remark}

\newcommand{\fraka}{\mathfrak a}

\newcommand{\frakf}{\mathfrak f}

\newcommand{\frakp}{\mathfrak p}

\newcommand{\la}{\lambda}

\newcommand{\La}{\Lambda}
\newcommand{\Z}{\mathbb{Z}}
\newcommand{\A}{\mathbb{A}}
\newcommand{\N}{\mathbb{N}}
\newcommand{\C}{\mathbb{C}}
\newcommand{\K}{\mathbb{K}}

\newcommand{\F}{\mathbb{F}}

\newcommand{\cc}{\mathcal{C}}

\begin{document}

\title[Explicit formulas for Drinfeld modules]
{Explicit formulas for Drinfeld modules \\ and their periods}

\author{Ahmad El-Guindy}
\address{Current address: Science Program, Texas A\&M University in Qatar, Doha, Qatar}
\address{Permanent address: Department of Mathematics, Faculty of Science, Cairo University, Giza, Egypt 12613}
\email{a.elguindy@gmail.com}

\author{Matthew A. Papanikolas}
\address{Department of Mathematics, Texas A\&M University, College Station, TX 77843, U.S.A.}
\email{map@math.tamu.edu}

\thanks{Research of the second author was partially supported by NSF Grant DMS-0903838.}
\date{December 22, 2011}

\begin{abstract}
We provide explicit series expansions for the exponential and logarithm functions attached to a rank $r$ Drinfeld module that generalize well known formulas for the Carlitz exponential and logarithm.  Using these results we obtain a procedure and an analytic expression for computing the periods of rank 2 Drinfeld modules and also a criterion for supersingularity.
\end{abstract}

\maketitle

\section{Introduction}
The goal of this paper is to determine explicit formulas for exponential functions, logarithms, and periods of Drinfeld modules.  Originally Carlitz~\cite{Carlitz35} and Wade~\cite{Wade46} worked out a complete picture for the Carlitz module by deducing closed formulas for both the power series expansions of its associated exponential and logarithm functions and the Carlitz period.  Building on work of Hayes~\cite{Hayes79} on sgn-normalized Drinfeld modules of rank~$1$, Gekeler~\cite{GekelerBook} determined formulas for periods of rank~$1$.  Later Thakur~\cite{Thakur92}, \cite{Thakur93}, building on Hayes' work, found explicit formulas for exponentials and logarithms of rank~$1$ modules using special values of shtuka functions.  (See also \cite[Ch.~3, 4, 7]{GossBook} and \cite[Ch.~2, 8]{ThakurBook} for more information.)  In the current paper we have attempted to continue these investigations in a similar spirit for Drinfeld modules of arbitrary rank by developing a combinatorial framework of ``shadowed partitions'' to keep track of coefficient data.  As a result we obtain explicit formulas for the exponential and logarithm functions and for periods, and we further obtain a precise criterion for supersingularity that complements previous work of Cornelissen~\cite{Cornelissen99a}, \cite{Cornelissen99b} and Gekeler~\cite{Gekeler88}.

Let $q$ be a power of a prime $p$, and let $\F_q$ denote the field with $q$ elements. Consider the polynomial ring $\A=\F_q[T]$, and let $\K$ denote the fraction field of $\A$. Consider the unique valuation on $\K$ defined by
\[
v(T)=-1,
\]
which is the valuation at the ``infinite prime'' of the ring $\A$. Let $\K_\infty$ denote the completion of $\K$ with respect to $v$, and let $\C_\infty$ denote the completion of an algebraic closure of  $\K_\infty$. It is well known that $v$ has a unique extension to $\C_\infty$ that we still denote by $v$, and that $\C_\infty$ is a complete algebraically closed field. For any integer $n \in \N=\{0,1,\dots\}$ we write
\begin{equation}\label{dn}
\begin{split}
[n]&:=T^{q^n}-T,\\
D_n&:=[n][n-1]^q[n-2]^{q^2}\cdots [1]^{q^{n-1}},\quad D_0:=1,\\
L_n&:=(-1)^n [n][n-1]\cdots[2][1],\quad L_0:=1.
\end{split}
\end{equation}
A field $L$ is called an \emph{$\A$-field} if there is a nonzero homomorphism $\imath:{\A}\to L$. Examples of such fields are  extensions of either $\K$ or $\A/\mathfrak{p}$, where $\mathfrak{p}$ is a nonzero prime ideal of $\A$. For simplicity we will write $a$ in place of $\imath(a)$ when the context is clear.  Such a field has a \emph{Frobenuis homomorphism}
\[
\begin{split}
\tau:L&\to L\\
z&\mapsto z^q,
\end{split}
\]
and we can consider the ring $L\{\tau\}$ of polynomials in $\tau$ under addition and composition. Thus $\tau \ell=\ell^q\tau$ for any $\ell \in L$. A \emph{Drinfeld module of rank $r$ over $L$} is an $\F_q$-linear ring homomorphism $\phi:\A[T]\to L\{\tau\}$ such that
\begin{equation}
\phi_T=T+\sum_{i=1}^r A_i\tau^i, \quad A_i\in L,\quad A_r\neq 0.
\end{equation}
It then follows that the constant term of $\phi_a$ is $a$ for all $a \in \A$ and that the degree of $\phi_a$ in $\tau$ is $r\deg_T(a)$.

The simplest example of a Drinfeld module is the  \emph{Carlitz module} $\mathcal{C}$ given by
\[
\mathcal{C}_T=T+\tau,
\]
which has rank $1$.  Associated to the Carlitz module is the \emph{Carlitz exponential}
\begin{equation}\label{Cexp}
e_\cc(z):=\sum_{n=0}^\infty \frac{z^{q^n}}{D_n}.
\end{equation}
The series for $e_\cc$ converges for all $z\in \C_\infty$ and defines an entire, $\F_q$-linear, and surjective function. The key property connecting the Carlitz exponential to the Carlitz module is
\begin{equation}\label{CFE}
e_\cc(Tz)=\cc_T(e_\cc(z)),
\end{equation}
from which it follows that for all $a \in \A$,
\[
  e_\cc(az) = \cc_a(e_\cc(z)).
\]
The zeros of $e_\cc(z)$ form an $\A$-lattice of rank one in $\C_\infty$ with a certain generator $\pi_\cc\in \C_\infty$ called the \emph{Carlitz period}, and we have an alternate expression for the Carlitz exponential as
\begin{equation}\label{Clattice}
e_\cc(z)=z\prod_{0\neq\la\in \pi_\cc\A} \left(1-\frac{z}{\la}\right).
\end{equation}
It is useful to also consider the (local) composition inverse of $e_\cc$, called the \emph{Carlitz logarithm}, defined by
\begin{equation}\label{Clog}
\log_\cc(z):=\sum_{n=0}^\infty \frac{z^{q^n}}{L_n},\quad v(z)>\frac{-q}{q-1}.
\end{equation}
These results go back to Carlitz~\cite{Carlitz35} and Wade~\cite{Wade46}, who were investigating explicit class field theory over $\F_q(T)$.  See \cite[Ch.~3]{GossBook}, \cite[Ch.~2]{ThakurBook} for more details on the above constructions.

Analogues of \eqref{CFE}, \eqref{Clattice}, and \eqref{Clog} hold for any Drinfeld module over $\C_\infty$ (see \cite[Ch.~4]{GossBook}, \cite[Ch.~2]{ThakurBook} for more details).  Indeed in \cite{Carlitz95}, Carlitz himself had begun to study lattice functions for higher rank lattices long before Drinfeld developed the complete story.  In particular, if $\La\subset \C_\infty$ is an $\A$-lattice of rank $r$ then the \emph{lattice exponential function} defined by
\begin{equation}\label{lattice}
e_\La(z):=z\prod_{0\neq\la\in\La} \left(1-\frac{z}{\la}\right)
\end{equation}
is an entire, surjective, $\F_q$-linear function from $\C_\infty$ to $\C_\infty$ with kernel $\La$, and there exists a unique rank $r$ Drinfeld module $\phi=\phi(\La)$ such that
\begin{equation}\label{FE}
e_\La(Tz)=\phi_T(e_\La(z)).
\end{equation}
Furthermore, $e_\La$ has a series expansion of the form
\begin{equation}\label{expseries}
e_\La(z)=\sum_{n=0}^\infty\alpha_nz^{q^n}.
\end{equation}
It also has a local composition inverse $\log_\La$ with a series expansion
\begin{equation}\label{logseries}
\log_\La(z)=\sum_{n=0}^\infty\beta_nz^{q^n}.
\end{equation}
Note that the coefficients $\alpha_n\in\C_\infty$ (and consequently $\beta_n$) could be expressed in terms of $\La$ by expanding \eqref{lattice}: for instance
\begin{equation}\label{coeff}
\begin{split}
\alpha_1&=\sum_{\begin{split}\{\la_1,\dots,\la_{q-1}\}\subset \La\\
\la_i\neq \la_j\textrm{ for } i\neq j\end{split}}\prod_{i=1}^{q-1}\frac{1}{ \la_i},\\
\beta_1&=-\alpha_1,
\end{split}
\end{equation}
which is explicit, but rather complicated and impractical as it involves infinitely many terms. We can also describe the Drinfeld module $\phi(\La)$ in terms of $\La$ as follows. Start by noticing that $\La/T\La$ is a vector space of dimension $r$ over $\F_q$. Write
\begin{equation}
f(x):=\prod_{\la \in \La/T\La}\left(x-e_\La\left(\frac{\la}{T}\right)\right).
\end{equation}
It is well-known (see \cite[\S 4.3]{GossBook}) that $f(x)$ is $\F_q$-linear of degree $q^r$, hence of the form $f(x)=\sum_{n=0}^r A_n(\La)x^{q^n}$ for some $A_n(\La)\in \C_\infty$. Furthermore we have
\begin{equation}\label{module}
\phi_T(\La)=\sum_{n=0}^r A_n(\La)\tau^n.
\end{equation}

The general theme of the paper is in some sense to reverse the point of view of the previous paragraphs, and  provide explicit identities for the lattice $\La$, as well as the functions $e_\La$ and $\log_\La$, starting only from the knowledge of the Drinfeld module $\phi$. This is achieved by using relatively simple combinatorial objects that we name ``shadowed partitions'' which we introduce and study in \S 2. Using them, in \S 3 we obtain concrete formulas for $e_\La$ and $\log_\La$ (Theorem~\ref{alpha} and Theorem~\ref{beta}) that are as similar as could be hoped for to \eqref{Cexp} and \eqref{Clog}.  In \S 4 we restrict our attention to rank $2$ modules, and we proceed to study the convergence properties of $\log_\La$ making use of the detailed description we have for its coefficients (Corollary~\ref{logorder} and Corollary~\ref{logrange}).  In \S 5 we  examine the properties of $T$-torsion points of rank two Drinfeld modules, and show how, combined with the properties of $\log_\La$, we can recover at least one, and sometimes both generators of the lattice $\La$ in certain naturally defined ``families'' (Theorem~\ref{periods}), thus in some sense obtaining a converse of \eqref{module}.  In \S 6 we introduce additional conditions that enable us to obtain a completely analytic description for the period with maximal valuation (Theorem~\ref{maxval}).  In \S 7 we compare our results to an example of Thakur~\cite{Thakur92} arising from Drinfeld modules with complex multiplication.  Finally in \S 8 we study yet another application of shadowed partitions, where we introduce a ``multinomial'' theorem for any rank $r$ Drinfeld module (Theorem~\ref{multinomialthm}) and obtain as a consequence a concrete condition for supersingularity of a rank $2$ Drinfeld module at a prime $\frakp\in \A$ of any degree (Corollary~\ref{supersingular}).

\textbf{A note on notation.} In order to emphasize that our starting point is the  Drinfeld module rather than the lattice, from now on we shall write $e_\phi$, $\log_\phi$, and $A_n(\phi)$ instead of $e_\La$, $\log_\La$, and $A_n(\La)$, respectively.

\section{Shadowed partitions}
Recall that  a \emph{partition} of a set $S$ is a collection of subsets of $S$ that are pairwise disjoint,  and whose union is equal to $S$ itself. Also, if $S\subset \Z$, $j\in \Z$, then $S+j:=\{i+j: i\in S\}$.  For $r\in \N$ and $n \in \Z^+$, we set
\begin{multline}\label{prndefn}
P_r(n):=\bigl\{(S_1,\, S_2, \dots,\, S_r): S_i\subset \{0,\, 1, \dots,\, n-1\},\\
 \textnormal{and $\{S_i+j: 1\leq i \leq r,\, 0\leq j \leq i-1\}$ form a partition of  $\{0,\, 1, \dots,\, n-1\}$} \bigr\}.
\end{multline}
We also set $P_r(0):=\{\emptyset\}$ and $P_r(-n):=\emptyset$.
We propose to name elements of $P_r(n)$ as \emph{order~$r$ index-shadowed partitions of $n$}, or \emph{shadowed partitions} for short,  as each $S_i$ relies on its~$i$ ``shadows'' $S_i+j$ (including itself) so that all together they partition $n$ elements.
Furthermore, for $1 \leq i \leq r$ we set
\begin{equation}
P_r^i(n):=\{(S_1,\, S_2, \dots,\, S_r) \in P_r(n): 0 \in S_i\}.
\end{equation}
We collect some simple, yet important facts about these objects in the following lemma. Recall that the sequence of \emph{$r$-step Fibonacci numbers} $\{F^{(r)}_n\}$ is defined as follows ($n\in \Z^+$)
\begin{equation}
F^{(r)}_{-n}=0,\,\,\, F^{(r)}_0=1, \,\,\, F^{(r)}_n=\sum_{i=n-r}^{n-1} F^{(r)}_i.
\end{equation}

\begin{lemma} \label{facts1}
\begin{enumerate}
\item $\{P^i_r(n): 1 \leq i \leq r \}$ is a partition of $P_r(n)$.
\item For $1\leq i \leq r$, $P_r(n-i)$ could be identified with $P^i_r(n)$ via the well-defined bijection
\begin{equation}\label{map}
(S_1,\, S_2, \dots,\, S_r)\mapsto (S_1+i,\, S_2+i,\dots,\{0\}\cup(S_i+i),\dots, \, S_r+i).
\end{equation}
\item For all $r>0$ and all $n\in \Z$ we have $|P_r(n)|=F_n^{(r)}$.
\end{enumerate}
\end{lemma}

\begin{proof}
The proofs of (i) and (ii) consist of simple verifications that we leave to the reader, and statement (iii) follows from combining (i) and (ii).
\end{proof}

Let $S\subset \N$ be finite, and define the integer $w(S)$ by
\begin{equation}
w(S):=\sum_{i \in S} q^i.
\end{equation}
Note that $w(\emptyset)=0$.  To help simplify our formulas, we shall usually denote $(S_1,\dots, S_r)\in P_r(n)$ by ${\bf S}$. We fix the following notation for the rest of the paper
\[
|{\bf S}| :=\sum_{i=1}^r |S_i|,\quad
\bigcup {\bf S} :=\bigcup_{i=1}^r S_i,
\]
and
\[
{\bf S}+i :=(S_1+i,\dots,S_r+i)\in P_r(n+i), \, \textrm{ for } i\in \N.
\]
We collect some more facts that are relevant to our results.

\begin{lemma}
\begin{enumerate}
\item For ${\bf S} \in P_r(n)$, $\cup {\bf S}$ uniquely defines ${\bf S}$. Hence
\begin{equation}\label{rfib}
F_n^{(r)}=|P_r(n)|\leq 2^n.
\end{equation}
\item The $r$-tuple of sets $(S_1,\dots,S_r)$ is in $P_r(n)$ if and only if
\begin{equation}\label{qn-1}
\sum_{i=1}^r (q^i-1)w(S_i)=q^n-1.
\end{equation}
\end{enumerate}
\end{lemma}

\begin{proof}
To prove (i), write
\[
\cup {\bf S}=\{s_1, \, s_2,\dots, s_m\},
\]
with $s_i<s_{i+1}$. Note that the conditions in \eqref{prndefn} on $P_r(n)$ imply that $1\leq s_{i+1}-s_i \leq r$ and also $1\leq n-s_m \leq r$. It follows that we must have
\[
s_m \in S_{n-s_m},
\]
and, for $1\leq i\leq m-1$
\[
s_i \in S_{s_{i+1}-s_i}.
\]
Thus the sets $S_i$ are completely defined once we know $\cup{\bf S}$. It follows that the map \[\cup: P_r(n) \rightarrow \textrm { Subsets of }\{0,\dots, n-1\}\] is an injection, and \eqref{rfib} follows.  To prove \eqref{qn-1} divide both sides by $q-1$ to get
\begin{multline*}
w(S_1)+(q+1)w(S_2)+\dots + (q^{r-1}+q^{r-2}+\dots+q+1)w(S_r)\\
=q^{n-1}+q^{n-2}+\dots+q+1,
\end{multline*}
and the statement of (ii) follows.
\end{proof}

\section{Explicit formulas for lattice functions}
Let $\phi$ be a Drinfeld module of rank $r$ over $\C_\infty$. The corresponding exponential function $e_\phi$ on $\C_\infty$ satisfies
\begin{equation}\label{expFE}
e_\phi(Tz)=\phi_T(e_\phi(z)).
\end{equation}
It has the series expansion
\[
e_\phi(z)=\sum_{n=0}^\infty \alpha_n z^{q^n},
\]
with $\alpha_n=\alpha_n(\phi) \in \C_\infty$ and $\alpha_0=1$. In the next theorem we give an explicit formula for $\alpha_n$ in terms of the coefficients of $\phi$, and thus we give a proof of the existence of $e_\phi$ different from \eqref{lattice}. For notational convenience, for $(A_1,\dots, A_r)\in \C_\infty^r$ and ${\bf S}\in P_r(n)$ we write
\begin{equation}
{\bf A^S}:=\prod_{i=1}^r A_i^{w(S_i)}.
\end{equation}
Note that ${\bf A}^\emptyset=1$.

\begin{theorem}\label{alpha}
Let $\phi$ be a rank $r$ Drinfeld module given by
\[
\phi_T=\sum_{i=0}^r A_i \tau^i,\quad A_i\in \C_\infty.
\]
For $n\geq 0$ and for any $S\subset \{0,\, 1, \dots, n-1\}$ set
\begin{equation}
D_n(S):=\prod_{i \in S} [n-i]^{q^{i}}.
\end{equation}
If we set
\begin{equation}\label{alphaformula}
\alpha_n=\sum_{{\bf S}\in P_r(n)} \frac{{\bf A^S}}{D_n(\cup {\bf S})},
\end{equation}
then the series $\sum_{n=0}^\infty \alpha_nz^{q^n}$ converges on $\C_\infty$ and is the unique solution to \eqref{expFE} with $\alpha_0=1$ and thus $\exp_\phi(z) = \sum_{n=0}^\infty \alpha_nz^{q^n}$.
\end{theorem}

\begin{proof}
The functional equation
\[
e_\phi(Tz)=Te_\phi(z)+\sum_{i=1}^r A_i e_\phi(z)^{q^i}
\]
is equivalent to the  recursion
\begin{equation}\label{alpharecursion}
\alpha_nT^{q^n}=\sum_{i=0}^rA_i \alpha_{n-i}^{q^i}.
\end{equation}
For convenience, we can set $\alpha_n = 0$ for $n< 0$ so that \eqref{alpharecursion} holds for all $n \geq 0$.  We proceed by induction on $n$ to show that \eqref{alphaformula} is the unique solution to \eqref{alpharecursion} with $\alpha_0=1$. Since $D_n(\emptyset)=1$, formula \eqref{alphaformula} indeed gives $\alpha_0=1$. Substituting $A_0=T$,  the induction hypothesis gives
\[
\begin{split}
\alpha_n&=\frac{1}{[n]} \sum_{i=1}^r A_i \sum_{{\bf S}\in P_r(n-i)}\left(\frac{{\bf A^S}}{D_{n-i}(\cup {\bf S})}\right)^{q^i}\\
&=\frac{1}{[n]} \sum_{i=1}^r A_i \sum_{{\bf S}\in P_r(n-i)}\frac{{\bf A}^{{\bf S}+i}}{D_{n}(\cup {\bf S}+i)}\\
&=\sum_{i=1}^r \sum_{{\bf S}\in P^i_r(n)}\frac{{\bf A^S}}{D_{n}(\cup {\bf S})},
\end{split}
\]
where the last equality follows from Lemma \ref{facts1}(ii) and the fact that
\[
D_{n-i}(S)^{q^i}=D_n(S+i).
\]
Thus \eqref{alphaformula} is proved.

We could deduce the convergence of $\sum \alpha_nz^{q^n}$ on all of $\C_\infty$ by relying on the corresponding property of the lattice exponential function given by \eqref{lattice}. However, to emphasize that our approach suffices to develop important aspects of the theory, we use \eqref{alphaformula} to give a direct proof from first principles. Note that $v([n-i]^{q^i})=-q^n$, and thus $v(D_n(S))=-q^n|S|$. It follows that for ${\bf S} \in P_r(n)$ we have
\begin{equation}\label{vDn}
v({D_n(\cup{\bf S})})\leq -\frac{nq^n}{r}.
\end{equation}
Next, set $v_0=\min_{1\leq i\leq r}v(A_i)$. It is easy to see that for ${\bf S} \in P_r(n)$ we have
\begin{equation}\label{vAS}
v({\bf A^S})\geq w(\cup{\bf S})v_0.
\end{equation}
For ${\bf S}\in P_r(n)$ we have
\[
\frac{q^n-1}{q^r-1}\leq w(\cup{\bf S})\leq \frac{q^n-1}{q-1}.
\]
Together with \eqref{vDn} and \eqref{vAS} we get
\begin{equation}\label{lim}
v(\alpha_n z^{q^n})\geq \begin{cases}
q^n\left(\frac{n}{r}+\frac{v_0}{q-1}+v(z)\right), \, &\textrm{ if } v_0<0, \\
q^n\left(\frac{n}{r}+v(z)\right), \, &\textrm{ if } v_0\geq 0,
\end{cases}
\end{equation}
and it follows that
\[
\lim_{n\to +\infty}v(\alpha_nz^{q^n})= +\infty.
\]
Hence the series converges for all $z\in \C_\infty$.
\end{proof}

\begin{example}
We write a few concrete cases to clarify \eqref{alphaformula}. Let the superscript on $\alpha$ indicate the rank $r$ of the corresponding module. Then for $r=2$ we get, for instance
\begin{gather*}
\alpha_3^{(2)}=\frac{A_1^{q^2+q+1}}{[1]^{q^2}[2]^{q}[3]}+\frac{A_2^qA_1}{[2]^q[3]}
+\frac{A_1^{q^2}A_2}{[1]^{q^2}[3]}, \\
\alpha_4^{(2)}=\frac{A_1^{q^3+q^2+q+1}}{[1]^{q^3}[2]^{q^{2}}[3]^q[4]}
+\frac{A_2^qA_1^{q^3+1}}{[2]^{q^2}[3]^q[4]}+\frac{A_1^{q^3+q^2}A_2}{[1]^{q^3}[3]^q[4]}
+\frac{A_2^{q^2}A_1^{q+1}}{[2]^{q^2}[3]^q[4]}+\frac{A_2^{q^2+1}}{[2]^{q^2}[4]},
\end{gather*}
whereas for $r=3$ we get
\[
\alpha_3^{(3)}=\frac{A_1^{q^2+q+1}}{[1]^{q^2}[2]^{q^{1}}[3]}+\frac{A_2^qA_1}{[2]^q[3]}+\frac{A_1^{q^2}A_2}{[1]^{q^2}[3]}+\frac{A_3}{[3]},
\]
\begin{multline*}
\alpha_4^{(3)}=\frac{A_1^{q^3+q^2+q+1}}{[1]^{q^3}[2]^{q^{2}}[3]^q[4]}
+\frac{A_2^qA_1^{q^3+1}}{[2]^{q^2}[3]^q[4]}+\frac{A_1^{q^3+q^2}A_2}{[1]^{q^3}[3]^q[4]}
+\frac{A_2^{q^2}A_1^{q+1}}{[2]^{q^2}[3]^q[4]}\\
+\frac{A_2^{q^2+1}}{[2]^{q^2}[4]}
+\frac{A_3A_1^{q^3}}{[1]^{q^3}[4]}+\frac{A_3^qA_1}{[3]^{q}[4]}.
\end{multline*}
\end{example}

Next we study the function $\log_\phi$. From \eqref{expFE} we obtain the functional equation
\begin{equation}\label{logFE}
T\log_\phi(z)=\log_\phi(\phi_T(z)).
\end{equation}
The following proposition provides a concrete description of the coefficients of $\log_\phi$.
\begin{theorem} \label{beta}
Given a Drinfeld module $\phi$ of rank $r$, write
\[
\log_\phi(z)=\sum_{n=0}^\infty \beta_nz^{q^n}.
\]
For ${\bf S} \in P_r(n)$ set
\begin{equation}\label{LS}
L({\bf S}):=\prod_{j=1}^r \prod_{i \in S_j} (-[i+j]).
\end{equation}
Then
\begin{equation}\label{betaformula}
\beta_n=\sum_{{\bf S} \in P_r(n)}\frac{{\bf A^S}}{L({\bf S})}.
\end{equation}
\end{theorem}

\begin{proof}
Since $L(\emptyset)=1$, we see that $\beta_0=1$, as expected. Now the functional equation \eqref{logFE} gives the recursion
\begin{equation}\label{tbetan}
T\beta_n=\sum_{i=0}^r \beta_{n-i}A_{i}^{q^{n-i}},
\end{equation}
where again we set $\beta_n=0$ for $n < 0$.  Applying $A_0=T$ and the induction hypothesis gives
\begin{equation}\label{betastep}
\begin{split}
-[n]\beta_n&=\sum_{i=1}^r A_i^{q^{n-i}}\sum_{{\bf S}\in P_r(n-i)}\frac{{\bf A^S}}{L({\bf S})}.\\
\end{split}
\end{equation}
Note that for $1\leq i\leq r$, the map $\Psi_i:P_r(n-i)\rightarrow P_r(n)$ given by
\[
\Psi_i(S_1,\, S_2,\dots, S_i,\dots, S_r)=(S_1, \, S_2, \dots, S_i\cup\{n-i\}, \dots, S_r),
\]
is a well-defined injection, and furthermore the collection $\{\Psi_i(P_r(n-i)): 1\leq i\leq r\}$ is a partition of $P_r(n)$. Also, note that for all $1\leq i\leq r$,
\[
L(\Psi_i{\bf S})= -[n]\cdot L({\bf S}).
\]
Thus from \eqref{betastep} we get
\[
\beta_n=\sum_{i=1}^r \sum_{{\bf S}\in \Psi_i(P_{r}(n-i))}\frac{{\bf A^S}}{L({\bf S})},
\]
and the result follows.
\end{proof}

\begin{remark}
When $\phi$ is the Carlitz module $\mathcal{C}$, we recover \eqref{Cexp} and \eqref{Clog} from \eqref{alphaformula} and \eqref{betaformula}, respectively.
\end{remark}

\begin{remark}
If we assign to each $A_n$ a ``weight" of $q^n-1$, and extend it to products in the usual way ($\textrm{wt}(AB)=\textrm{wt}(A)+\textrm{wt}(B)$), then by \eqref{qn-1} the total weight of any ${\bf A^S}$ appearing as a summand in the coefficient of $z^{q^n}$ in either $e_\phi$ or $\log_\phi$ is always $q^n-1$.
\end{remark}

\section{Convergence and range of $\log_\phi$ in rank two}
Let $\phi$ be a rank $2$ Drinfeld module given by
\begin{equation}\label{rank2}
\phi_T=T+A\tau+B\tau^2.
\end{equation}
The \emph{$\jmath$-invariant} of $\phi$ is defined by
\begin{equation}\label{jinv}
\jmath(\phi):=\frac{A^{q+1}}{B}.
\end{equation}
In this section we  study the convergence properties of the series defining $\log_\phi$. We start with determining the valuation of its coefficients.

\begin{lemma}\label{vbeta}
Let $\phi$ be as in \eqref{rank2} and write
\[
\log_\phi(z)=\sum_{n=0}^\infty \beta_nz^{q^n}.
\]
Then for $n\in \N$, $v(\beta_n)$ is given by the following formula.
\begin{equation}\label{betaval}
v(\beta_n)=\begin{cases}
{\displaystyle \frac{q^n-1}{q-1}(v(A)+q)},& \text{if $v(\jmath)<-q$,}\\[10pt]
{\displaystyle \frac{q^n-1}{q^2-1}(v(B)+q^2)},&  \text{if $v(\jmath)>-q$ and $n$ is even,}\\[10pt]
{\displaystyle \frac{q^n-1}{q^2-1}(v(B)+q^2)+\frac{v(\jmath)+q}{q+1}},& \text{if $v(\jmath)>-q$ and  $n$ is odd}.
\end{cases}
\end{equation}
In the case where $v(\jmath)=-q$ we have
\begin{equation}\label{betainequality}
v(\beta_n)\geq \frac{q^n-1}{q-1} (v(A)+q)=\frac{q^n-1}{q^2-1}(v(B)+q^2),
\end{equation}
with equality holding infinitely often.
\end{lemma}

\begin{proof}
From \eqref{betaformula} we have
\[
\beta_n=\sum_{(S_1, S_2)\in P_2(n)} \frac{A^{w(S_1)}B^{w(S_2)}}{L(S_1,S_2)},
\]
where
\[
L(S_1,S_2)=\prod_{i\in S_1}(-[i+1]) \prod_{i\in S_2} (-[i+2]).
\]
It is easy to  see that
\[
v\left(\frac{A^{w(S_1)}B^{w(S_2)}}{L(S_1,S_2)}\right)=w(S_1)v(A)+w(S_2)v(B)+qw(S_1)+q^2w(S_2).
\]
In addition, by \eqref{qn-1} we have $(q-1)w(S_1)+(q^2-1)w(S_2)=q^n-1$, hence
\[
w(S_1)=\frac{q^n-1}{q-1}-(q+1)w(S_2),
\]
and consequently
\begin{equation}\label{betamainval}
v\left(\frac{A^{w(S_1)}B^{w(S_2)}}{L(S_1,S_2)}\right)=\frac{q^n-1}{q-1}(v(A)+q)-w(S_2)(v(\jmath)+q).
\end{equation}
Thus our analysis naturally breaks into the following three cases.

\textbf{Case 1: $v(\jmath)>-q$.} In this case, we see that $v\left(\frac{A^{w(S_1)}B^{w(S_2)}}{L(S_1,S_2)}\right)$ is a strictly decreasing function of $w(S_2)$, and hence attains its minimal value when $w(S_2)$ is maximal. It is easy to see that
\begin{equation}
\max_{(S_1,S_2) \in P_r(n)}(w(S_2))= \begin{cases}
{\displaystyle \frac{q^n-1}{q^2-1}},  &\textnormal{if $n$ is even,}\\[10pt]
{\displaystyle \frac{q^n-q}{q^2-1}},  &\textrm{if $n$ is odd. }
\end{cases}
\end{equation}
(Corresponding to $S_1=\emptyset$ for $n$ even and $S_1=\{0\}$ for odd $n$). The ultrametric property implies
\begin{equation}
v(\beta_n)=\begin{cases}
{\displaystyle \frac{q^n-1}{q-1}(v(A)+q)-\frac{q^n-1}{q^2-1}(v(\jmath)+q)}, &\textnormal{if $n$ is even,}\\[10pt]
{\displaystyle \frac{q^n-1}{q-1}(v(A)+q)-\frac{q^n-q}{q^2-1}(v(\jmath)+q)}, &\textnormal{if $n$ is odd,}
\end{cases}
\end{equation}
and the corresponding part of \eqref{betaval} follows.

\textbf{Case 2: $v(\jmath)<-q$.} From \eqref{betamainval} we see that $v\left(\frac{A^{w(S_1)}B^{w(S_2)}}{L(S_1,S_2)}\right)$ is strictly increasing in $w(S_2)$. Thus the minimal valuation is attained when $S_2=\emptyset$, which implies the first line of \eqref{betaval}.

\textbf{Case 3: $v(\jmath)=-q$.} In this case we see that $v\left(\frac{A^{w(S_1)}B^{w(S_2)}}{L(S_1,S_2)}\right)$ is always equal to
\[
\frac{q^n-1}{q-1} (v(A)+q)=\frac{q^n-1}{q^2-1}(v(B)+q^2),
\]
and hence $v(\beta_n)\geq \frac{q^n-1}{q-1}(v(A)+q)$. It is easy to see by direct calculation that we have equality for $n=0,1$. The recurrence formula
\[
-[n]\beta_n=A^{q^{n-1}}\beta_{n-1}+B^{q^{n-2}}\beta_{n-2}
\]
implies that
\[
v(\beta_n)-q^n\geq \inf\left(q^{n-1}v(A)+v(\beta_{n-1}),q^{n-2}v(B)+v(\beta_{n-2})\right).\]
Since $v(\jmath)=-q$, an easy computation shows that
\begin{multline*}
q^{n-1}v(A)+\frac{q^{n-1}-1}{q-1}(v(A)+q)\\ =q^{n-2}v(B)+\frac{q^{n-2}-1}{q-1}(v(A)+q)
=\frac{q^n-1}{q-1}(v(A)+q)-q^n.
\end{multline*}
Thus if $v(\beta_n)>\frac{q^n-1}{q-1}(v(A)+q)$, while $v(\beta_{n-1})=\frac{q^{n-1}-1}{q-1}(v(A)+q)$, then we must have $v(\beta_{n+1})=\frac{q^{n+1}-1}{q-1}(v(A)+q)$ and also $v(\beta_{n+2})=\frac{q^{n+2}-1}{q-1}(v(A)+q)$. Thus equality in \eqref{betainequality} actually occurs at least two thirds of the time, and the result follows.
\end{proof}

\begin{corollary}\label{logorder}
Set
\begin{equation}
\begin{split}
\rho_B&:=-\frac{q^2+v(B)}{q^2-1},\\
\rho_A&:=-\frac{q+v(A)}{q-1}.
\end{split}
\end{equation}
If $v(\jmath)\geq-q$ then the series $\sum \beta_iz^{q^i}$ converges exactly for $z\in \C_\infty$ with   $v(z)>\rho_B$, and if $v(\jmath)\leq-q$ then it converges exactly for  $v(z)>\rho_A$.
\end{corollary}

\begin{proof}
We know that the series converges if and only if $\lim_{n\to +\infty}v(\beta_nz^{q^n})=+\infty$. From Lemma~\ref{vbeta} we have
\begin{equation}\label{orderlog}
v(\beta_nz^{q^n})= \begin{cases}
{ (q^n-1)\left(\frac{v(A)+q}{q-1}+v(z)\right)+v(z)}, &\textnormal{if $v(\jmath)<-q$,}\\[10pt]
{ (q^n-1)\left(\frac{v(B)+q^2}{q^2-1}+v(z)\right)+v(z)}, &\parbox{1truein}{if $v(\jmath)>-q$ \\ and $n$ is even,}\\[10pt]
{(q^n-1)\left(\frac{v(B)+q^2}{q^2-1}+v(z)\right)+v(z)+\frac{v(\jmath)+q}{q+1}}, &\parbox{1truein}{if $v(\jmath)>-q$ \\ and $n$ is odd.}\\
\end{cases}
\end{equation}
When $v(\jmath)=-q$ then we have
\begin{equation}\label{manyn}
v(\beta_n z^{q^n})\geq(q^n-1)\left(\frac{v(B)+q^2}{q^2-1}+v(z)\right)+v(z),
\end{equation}
with equality holding for infinitely many values of $n$. The result follows at once.
\end{proof}

\begin{corollary}\label{logrange}
If the series for $\log_\phi$ converges for $z\in \C_\infty$ then we must have
\begin{equation}
v(\log_\phi(z))=v(z).
\end{equation}
\end{corollary}

\begin{proof}
From \eqref{orderlog} and \eqref{manyn}, it is easy to see that in the case of convergence we must have
\[
v(z)<v(\beta_nz^{q^n}) \textrm{ for all } n\geq 1,
\]
and the result follows by the ultrametric property of $v$.
\end{proof}

\begin{remark} Note that $\rho_B-\rho_A=\frac{v(\jmath)+q}{q^2-1}$. If we set
\[
\rho_\phi:=\max(\rho_A,\rho_B),
\]
then we can rephrase Corollary \ref{logorder} by saying that the series of $\log_\phi$ converges at $z\in \C_\infty$ if and only if $v(z)>\rho_\phi$.
\end{remark}

\section{Computing the periods of rank two Drinfeld modules}
Let $\phi$ be a Drinfeld module given by
\[
\phi_T=T+A\tau+B\tau^2,
\]
and let $\La_\phi$ be the corresponding lattice.
We can describe $\La_\phi$ as the unique lattice for which
\begin{equation}
e_{\La_\phi}=e_\phi.
\end{equation}
In other words, $\La_\phi$ is the set of zeros of $e_\phi$.
Our goal in this section is to outline a procedure for obtaining \emph{periods} of $\phi$, (i.e. elements of the lattice $\La_\phi$) in terms of the coefficients $A$ and $B$.
We start with an easy lemma on the values of $e_\phi$ at the $T$-division points of $\phi$.
\begin{lemma}\label{Tdivision}
For every $\la \in \La_\phi$, $\delta_\la:=e_\phi\left(\frac{\la}{T}\right)$ is a root of the polynomial
\[
Bx^{q^2}+Ax^{q}+Tx=0.
\]
\end{lemma}

\begin{proof}
The function $e_\phi$ satisfies the functional equation $e_\phi(Tz)=\phi_T(e_\phi(z))$.
Hence
\[
0=e_\phi\left(T\cdot \frac{\la}{T}\right)=T\delta_\la+A\delta_\la^q+B\delta_\la^{q^2},
\]
and the result follows.
\end{proof}
For the remainder of the paper, we let
\begin{equation}
f_\phi(x)=Bx^{q^2}+Ax^q+Tx.
\end{equation}
Also set
\begin{align*}
V_\phi&:=\{\delta \in \C_\infty: f_\phi(\delta)=0\},\\
V_\phi^*&:=\{\delta \in \C_\infty: B\delta^{q^2-1}+A\delta^{q-1}+T=0\}.
\end{align*}
As $V_\phi$ is the $T$-torsion submodule on $\phi$ it follows that $V_\phi$ is a $2$-dimensional vector space over~$\F_q$, and $V_\phi^*$ is its set of nonzero elements. The following lemma gives a complete description of the possible valuations on $V^*_\phi$.

\begin{lemma}\label{vdelta}
Let $\phi$ be a rank $2$ Drinfeld  module given by \eqref{rank2}, and let $\jmath$ be its $\jmath$-invariant as in \eqref{jinv}. Exactly one of the following cases hold.
\begin{enumerate}
\item  All the elements of $V_\phi^*$ have the same valuation given by
\begin{equation}\label{allsame}
v(\delta)=\frac{-(1+v(B))}{q^2-1} \textrm{ for all } \delta \in V_\phi^*.
\end{equation}
This case happens  if and only if $v(\jmath)\geq -q$.
\medskip
\item There is  an element $\eta\in V_\phi^*$ such that all elements of  $\F_q^* \eta$ have strictly larger valuation than the rest of $V_\phi^*$ if an only if $v(\jmath)<-q$. In this case we have
\begin{equation}\label{bigone}
v(\eta)=\frac{-(1+v(A))}{q-1} \textrm{ and }  v(\delta)=\frac{v(A)-v(B)}{q^2-q} \textrm{ for all } \delta \in V_\phi\setminus \F_q\eta.
\end{equation}
\end{enumerate}
\end{lemma}

\begin{proof}
The lemma follows from an analysis of the Newton polygon of the defining polynomial $f_\phi(x)/x = Bx^{q^2-1} + Ax^{q-1} + T = 0$ of $V_\phi^*$ \cite[Ch.~2]{GossBook}, \cite[\S I.2]{KedlayaBook}.  Indeed the line segment connecting $(0,-1)$ and $(q^2-1,v(B))$ has slope $\frac{v(B) + 1}{q^2-1}$, and one checks that $(q-1,v(A))$ lies on or above this line segment if and only if $v(\jmath) = (q+1)v(A) - v(B) \geq -q$.  Thus $v(\jmath) \geq -q$ if and only if all zeroes of $f_\phi(x)/x$ have valuation $-\frac{v(B)+1}{q^2-1}$.  Otherwise, when $v(\jmath) < -q$, the Newton polygon breaks into two segments: one of width $q-1$ from $(0,-1)$ to $(q-1,v(A))$ of slope $\frac{v(A)+1}{q-1}$, and another of width $q^2-q$ from $(q-1,v(A))$ to $(q^2-1,v(B))$ of slope $\frac{v(B)-v(A)}{q^2-q}$.  The result then follows.
\end{proof}

Guided by the results above, we consider certain families of rank two Drinfeld modules as follows. Fix $0\neq\delta\in \C_\infty$, and set
\begin{equation}
\mathcal{F_\delta}:=\{\textrm{All rank 2 Drinfeld modules $\phi$ such that }\F_q\delta\subset V_\phi\}.
\end{equation}
The following theorem gives a complete description of the cases where the lattice $\La_\phi$ could be recovered by applying $\log_\phi$ to $V_\phi$.

\begin{theorem}\label{periods}
Let $\phi\in\mathcal{F_\delta}$ be given by $\phi_T=T+A\tau+B\tau^2$. Fix a choice of a $(q-1)$-st root of $\frac{T}{B}$ and set
\begin{equation}\label{c}
c:=\delta^{-1}\left(\frac{T}{B}\right)^\frac{1}{q-1}.
\end{equation}
Let $\zeta$ be a root of
\begin{equation}\label{delta}
x^q-\delta^{q-1}x=c.
\end{equation}
We have the following cases.
\begin{enumerate}
\item If $v(\jmath) \geq -q$, then $v(\delta)=v(\zeta)=\frac{-(1+v(B))}{q^2-1}$. Hence $\log_\phi$ converges at $\delta$ and $\zeta$, and the period lattice $\La_\phi$ is generated by $\{T\log_\phi(\delta),T\log_\phi(\zeta)\}$.
\medskip
\item If $v(\jmath)<-q$ and $v(\delta)=\frac{-(1+v(A))}{q-1}$, then $\log_\phi$ converges at $\delta$, and $T\log_\phi(\delta)$ is a period in $\La_\phi$. Furthermore $v(\zeta)=\frac{v(A)-v(B)}{q^2-q}$, and $\log_\phi$ converges on all of $V_\phi$ if and only if
\begin{equation}\label{jq2}
v(\jmath)>-q^2.
\end{equation}
If \eqref{jq2} is satisfied then the period lattice $\La_\phi$ is generated by $\{T\log_\phi(\delta),T\log_\phi(\zeta)\}$.
\medskip
\item If $v(\jmath)<-q$ and $v(\delta)=\frac{v(A)-v(B)}{q^2-q}$, then $v(\zeta)=\frac{-(1+v(A))}{q-1}$ and $\log_\phi$ converges at $\zeta$, hence $T\log_\phi(\zeta)$ is a period in $\La_\phi$. Again $\log_\phi$ converges on all of $V_\phi$ if and only \eqref{jq2} is satisfied, in which case
 the period lattice $\La_\phi$ is generated by $\{T\log_\phi(\delta),T\log_\phi(\zeta)\}$.
\end{enumerate}
\end{theorem}

\begin{proof}
The condition $\phi\in \mathcal{F}_\delta$ is equivalent to
\begin{equation}\label{delta1}
B\delta^{q^2}+A\delta^q+T\delta=0.
\end{equation}
Substituting \eqref{delta1} in $f_\phi(x)$ we get

\[
\begin{split}
f_\phi(x)&=Bx^{q^2}-(T\delta^{1-q}+B\delta^{q^2-q})x^q+Tx\\
&=B(x^q-\delta^{q-1}x)^q-T\delta^{1-q}(x^q-\delta^{q-1}x).
\end{split}
\]
It follows that any $\zeta\in V_\phi\setminus \F_q\delta$ must satisfy
\begin{equation}\label{zeta}
\zeta^q-\delta^{q-1}\zeta=\delta^{-1}\left(\frac{T}{B}\right)^{\frac{1}{q-1}}.
\end{equation}
Obviously $\frac{-(1+v(B)}{q^2-1}>\rho_B$, and it follows that when $v(\jmath)\geq -q$, $\log_\phi$ converges on all of $V_\phi$. When $v(\jmath)<-q$, we also have $\frac{-(1+v(A))}{q-1}>\rho_A$; however we have
\[
\frac{v(A)-v(B)}{q^2-q}>\rho_A=-\frac{q+v(A)}{q-1} \textrm { if and only if } v(\jmath)>-q^2.
\]
Finally assume that $\delta$ and $\zeta$ are linearly independent over $\F_q$, and that $\log_\phi$ converges at both of them. We need to show that $\log_\phi(\delta)$ and $\log_\phi(\zeta)$ are linearly independent over $\A$. From Lemma \ref{Tdivision} we see that $e_\phi(T^n\log_\phi\eta)=0$ for all $n\geq 1$ and all $\eta\in V_\phi$. Thus if $a,b \in \A$ are polynomials with constant terms $a_0$ and $b_0$ respectively, then
\[
e_\phi(a\log_\phi(\delta)+b\log_\phi(\zeta))=a_0\delta+b_0\zeta,
\]
and it follows that indeed $\{\log_\phi(\delta),\log_\phi(\zeta)\}$ are linearly independent over $\A$.
\end{proof}

\begin{remark} We note that \eqref{delta} could be written as
\begin{equation}\label{art-schr}
X^q-X=\frac{c}{\delta^q},
\end{equation}
where $X:=\frac{x}{\delta}$. Thus computing $\zeta$ is reduced to the extraction of an Artin-Schreier root.
\end{remark}

\section{An Analytic Expression for Periods}
In the previous section we obtained a procedure for computing periods which involved the extraction of certain roots. In this section we show that under additional conditions (cf.~\eqref{fstar}) we can obtain a completely analytic expression for the periods. We start with a lemma on expressing the roots of a certain algebraic equation in terms of series.

\begin{lemma}\label{equation}
Let $C,\delta \in \C_\infty\setminus\{0\}$. If
\begin{equation}\label{Cineq}
v(C)>qv(\delta)
\end{equation}
then the set of solutions of the equation
\begin{equation}\label{Ceqn}
x^q-\delta^{q-1}x=C
\end{equation}
is given by
\begin{equation}\label{sersol}
\F_q\delta-\delta\sum_{i=0}^\infty \left(\frac{C}{\delta^q}\right)^{q^i}.
\end{equation}
\end{lemma}

\begin{proof}
Condition \eqref{Cineq} guarantees the convergence of the infinite series. It can easily be seen that it satisfies \eqref{Ceqn}. Finally, notice that a polynomial of degree $q$ can have at most $q$ distinct solutions.
\end{proof}

\begin{corollary}\label{eta}
Let $\phi_T=T+A\tau+B\tau^2$ be a Drinfeld module with $v(\jmath)<-q$, and assume that $\delta\in V_\phi$ with $v(\delta)=\frac{v(A)-v(B)}{q^2-q}$. Fix a choice of a $(q-1)$-root of $\frac{T}{B}$. Then the unique subspace of $V_\phi$ where the valuation of the nonzero elements is $\frac{-(1+v(A))}{q-1}$ is generated by
\begin{equation}\label{etaseries}
\eta=-\delta \sum_{n=0}^\infty \left(\frac{T}{\delta^{q^2-1}B}\right)^\frac{q^n}{q-1}.
\end{equation}
\end{corollary}

\begin{proof}
With $c$ as in \eqref{c}, we see that
\[
v\left(\frac{c}{\delta^q}\right) = -(q+1)\frac{v(A)-v(B)}{q^2-q}-\frac{1+v(B)}{q-1}
=\frac{-(v(\jmath)+q)}{q^2-q}>0,
\]
and thus the series converges, and the valuation of the sum is equal to that of the first term by the ultrametric property. So indeed
\[
v(\eta)=v(\delta^{1-q}c)=\frac{-q(v(A)-v(B))}{q^2-q}-\frac{(1+v(B))}{q-1}=\frac{-(1+v(A))}{q-1},
\]
and the result follows from Lemma \ref{equation} and Lemma \ref{vdelta}.
\end{proof}

For $0\neq \delta \in \C_\infty$ we consider the subfamily $\mathcal{F}_\delta^\star$ of $\mathcal{F}_\delta$ defined by
\begin{equation}\label{fstar}
\mathcal{F}_\delta^\star:=\left\{\phi \in \mathcal{F}_\delta: v(\jmath)<-q \textrm { and } v(\delta)=\frac{v(A)-v(B)}{q^2-q}\right\}.
\end{equation}
We have the following analytic expression for periods in $\mathcal{F}_\delta^\star$.

\begin{theorem} \label{maxval}
Let $\phi \in \mathcal{F}_\delta^\star$ be given, and set $c$ as in \eqref{c}. Let $\beta_j$ be the coefficients of $\log_\phi$. Set
\begin{equation}\label{a}
\begin{split}
\fraka_\delta(n)&:=T\sum_{j=0}^n \beta_j\delta^{q^j}, \textrm{ and }\\
\frakf(z)&:=\sum_{n=0}^\infty \fraka_\delta(n)z^{q^n}.
\end{split}
\end{equation}
Then the series $\frakf$ converges for $z=\delta^{-q}c$, and $\frakf(\delta^{-q}c)$ is a period of $\La_\phi$ with maximal valuation.
\end{theorem}

\begin{proof}
Let $\eta$ be as in \eqref{etaseries}. From Theorem \ref{periods} and Corollary \ref{eta}, we see that a period $\la\in \La_\phi$ is given by
\begin{multline*}
\la := T\log_\phi(\eta)=T\log_\phi\left(\sum_{i=0}^\infty \delta^{1-q^{i+1}}c^{q^i}\right)
=T\sum_{j=0}^\infty \beta_j\sum_{i=0}^\infty \delta^{q^j-q^{i+j+1}}c^{q^{i+j}}\\
=\sum_{n=0}^\infty T\left(\sum_{j=0}^n \beta_j\delta^{q^j}\right) \left(\frac{c}{\delta^q}\right)^{q^n}
=\frakf (\delta^{-q}c).
\end{multline*}
By Corollary \ref{logrange} we see that
\[
v(\la)=-1-\frac{1+v(A)}{q-1}=\frac{-(q+v(A))}{q-1}.
\]
If $\lambda^\prime \in \Lambda_\phi$ has larger valuation, then $e_\phi(T^{-1}\lambda^\prime)$ is an element of $V_\phi$ with valuation larger than $\frac{-(1+v(A))}{q-1}$, which contradicts Lemma \ref{vdelta}, and the theorem follows.
\end{proof}

We end this section with a more detailed analysis of the function $\frakf$.

\begin{proposition}
Let $\phi \in \mathcal{F}_\delta^\star$ be given, and let $\fraka_\delta$ and $\frakf$ be as in \eqref{a}. If $-q>v(\jmath)> -q^2$ then $\frakf$ converges if and only if $v(z)>0$, and if $v(\jmath)<-q^2$ then $\frakf$ converges if and only if
\begin{equation}
v(z)>\frac{-(v(\jmath)+q^2)}{q^2-q}>0.
\end{equation}
If $v(\jmath)=-q^2$, then $\frakf$ converges at least for $v(z)>0$.
Furthermore for $n\geq 0$ we have
\begin{equation}
\fraka_\delta(n)=(T\delta)^{q^n}\beta_n-(B\delta^{q^2})^{q^{n-1}}\beta_{n-1},
\end{equation}
and in the range $v(z)>\frac{-(q+v(\jmath))}{q^2-q}$, $\frakf$ has the representation
\begin{equation}\label{frakf}
\frakf(z)=\log_\phi(T\delta z)-\log_\phi(B\delta^{q^2}z^q).
\end{equation}
\end{proposition}

\begin{proof}
{From~\eqref{orderlog}}, we see that
\[
\begin{split}
v(\beta_n\delta^{q^n})&=(q^n-1)\left(\frac{q+v(A)}{q-1}+\frac{v(A)-v(B)}{q^2-q}\right)+v(\delta)\\
&=(q^n-1)\left(\frac{v(\jmath)+q^2}{q^2-q}\right)+v(\delta).
\end{split}
\]
Thus our analysis naturally breaks into three cases.

\textbf{Case 1: $v(\jmath)>-q^2$.} In this case $v(T\beta_n\delta^{q^n})$ is strictly increasing in $n$, and thus
\[
v(\fraka_\delta(n))=v(\delta)-1 \textrm{ for all } n\geq 0.
\]
 It follows that the series for $\frakf$ converges if and only if $v(z)>0$.

\textbf{Case 2: $v(\jmath)<-q^2$.} In this case $v(T\beta_n\delta^{q^n})$ is strictly decreasing in $n$, and thus
\[
v(\fraka_\delta(n))=v(T\beta_n\delta^{q^n}).
\]
Hence
\[
\begin{split}
v(\fraka_\delta(n)z^{q^n})&=v(z)+v(\delta)-1+(q^n-1)\left(v(z)
+\frac{v(\jmath)+q^2}{q^2-q}\right),
\end{split}
\]
and $\frakf$ converges if and only if $v(z)>\frac{-(v(\jmath)+q^2)}{q^2-q}$.

\textbf{Case 3: $v(\jmath)=-q^2$.} In this case $v(T\beta_n\delta^{q^n})=v(\delta)-1$ for all $n$, and thus
\[
v(\fraka_\delta(n))\geq v(\delta)-1 \textrm{ for all } n\geq 0.
\]
It follows that the series for $\frakf$ converges at least for all $v(z)>0$.

The first part of the proposition follows from the analysis above. To prove the second part, note that  \eqref{tbetan} and \eqref{delta1} give
\[
\begin{split}
T\beta_i\delta^{q^i}&=(T\delta)^{q^i}\beta_i+(A\delta^q)^{q^{i-1}}\beta_{i-1}
+(B\delta^{q^2})^{q^{i-2}}\beta_{i-2}\\
&=(T\delta)^{q^{i}}\beta_i-(T\delta)^{q^{i-1}}\beta_{i-1}
-(B\delta^{q^2})^{q^{i-1}}\beta_{i-1}+(B\delta^{q^2})^{q^{j-2}}\beta_{i-2}.
\end{split}
\]
Hence
\[
\begin{split}
\fraka_\delta(n)&=(T\delta)^{q^n}\beta_n-(B\delta^{q^2})^{q^{n-1}}\beta_{n-1}.
\end{split}
\]
Consequently we (formally) get
\[
\frakf(z)=\log_\phi(T\delta z)-\log_\phi(B\delta^{q^2}z^q).
\]
The expression on the right hand side converges for  $v(T\delta z)>\rho_A$ and $v(B\delta^{q^2}z^q)>\rho_A$. Either statement is equivalent to
\[
v(z)>\frac{-(v(\jmath)+q)}{q^2-q},
\]
and the result follows.
\end{proof}

\begin{remark}
Note that since $v(\delta^{-q}c)=\frac{-(v(\jmath)+q)}{q^2-q}$, we can not use \eqref{frakf} to evaluate $\la=\frakf(\delta^{-q}c)$. Instead we can only use \eqref{a} for that evaluation, and indeed the argument above gives another proof that $\frakf$ does converge at that point.
\end{remark}

\section{An example from complex multiplication} \label{ThakurA}
In \cite{Thakur92}, Thakur determined the power series expansions of the exponential and logarithm functions of sgn-normalized rank~$1$ Drinfeld modules, and he showed how his constructions fit into the more general framework of shtuka functions in~\cite{Thakur93}, which is particular to the rank~$1$ theory.  Thakur's Drinfeld modules, originally studied by Hayes~\cite{Hayes79} in the context of explicit class field theory, are rank~$1$ over extensions of $\A$ but can be thought of as higher rank Drinfeld $\A$-modules with complex multiplication.  Here we consider one of Thakur's examples \cite[Ex.~A]{Thakur92} and compare it to the constructions of the previous sections.

Let $q=3$, and let $y \in \C_\infty$ satisfy $y^2 = T^3-T-1$.  Then define a rank $2$ Drinfeld module $\phi$ by setting
\begin{equation}
  \phi_T := T + y(T^3-T)\tau + \tau^2.
\end{equation}
The module $\phi$ is special in that it has complex multiplication by the ring $\F_3[T,y]$, which itself has class number one, and so $\phi$ is a rank $1$ Drinfeld $\F_3[T,y]$-module but a rank $2$ Drinfeld $\A$-module.
For $n \geq 1$, Thakur lets $[n]_y := y^{3^n}-y$ and sets
\[
  f_n = \frac{[n]_y - y[n]}{[n]-1}, \quad
  g_n = \frac{[n]_y - y^{3^n}[n]}{[n+1]+1}.
\]
He then establishes that if $e_\phi(z) = \sum \alpha_n z^{3^n}$ and $\log_\phi(z) = \sum \beta_n z^{3^n}$, then for $n \geq 1$,
\[
  \alpha_n = \frac{\alpha_{n-1}^3}{f_n}, \quad
  \beta_n = \frac{\beta_{n-1}}{g_n}.
\]
Therefore,
\[
  \alpha_n = \frac{1}{f_n f_{n-1}^3 \cdots f_1^{3^{n-1}}}, \quad
  \beta_n = \frac{1}{g_n g_{n-1} \cdots g_1}.
\]
After some calculations (and using that $v(y) = -\frac{3}{2}$), it follows that for $n \geq 1$,
\begin{align}
v(\alpha_n) &= {\textstyle \frac{1}{2}}(n-2) 3^n, \label{Thakalpha} \\
v(\beta_n) &= {\textstyle -\frac{3}{4}} (3^n-1). \label{Thakbeta}
\end{align}
See also Lutes~\cite[\S IV.C]{LutesThesis}. (In both Lutes and Thakur the formulas differ from the ones above by a factor of $\frac{1}{2}$, as they set $v(T)=-2$ instead of $-1$.)

Certainly the valuation of $\alpha_n$ in \eqref{Thakalpha} is consistent with \eqref{lim}, since $v_0 = -\frac{9}{2}$ in this case.   Now $\jmath(\phi) = y^4(T^3-T)^4$, and so $v(\jmath(\phi)) = -18$.  Therefore \eqref{Thakbeta} matches with Lemma~\ref{vbeta}, and $\log_\phi(z)$ converges for $v(z) > \frac{3}{4}$, which coincides with Corollary~\ref{logorder}.  Now the set $V_\phi$ is generated over $\F_3$ by $e_\phi(1/T)$ and $e_\phi(y/T)$, which have valuations $\frac{7}{4}$ and $-\frac{3}{4}$ respectively (see \cite[Ex.~4.15]{LutesP}), and thus $\phi$ fits into the situation of Theorem~\ref{periods} with $v(\jmath) < -q^2$ (so $\log_\phi$ does not converge on all of $V_\phi$).  However, one can also calculate the period by other means (see \cite[\S III]{GekelerBook}, \cite[\S 7.10]{GossBook}, \cite[Ex.~4.15]{LutesP}), and one finds that a period $\pi$ with maximal valuation has $v(\pi) = \frac{3}{4}$, agreeing with Theorem~\ref{maxval}.

\section{A multinomial formula and supersingular modules}
Let $\frakp \in \A$ be  monic and irreducible of degree $d$, and let $L_\frakp$ be a field extension of $\A/\frakp$.  A Drinfeld module $\phi$ of rank $r$ over $L_\frakp$ is said to be \emph{supersingular} if  $\phi_\frakp$ is purely inseparable.  If $\phi$ has rank $2$ then its supersingularity is equivalent to the vanishing of the coefficient of $\tau^d$ in $\phi_\frakp$ modulo $\frakp$. (See Gekeler \cite{Gekeler88}, \cite{Gekeler91} for further discussion and characterizations of supersingularity).

It is thus natural to seek a characterization of $\phi_{T^m}$ for all $m\in\N$ in terms of the coefficients of $\phi_T$. It turns out that the shadowed partitions of \S 2 above will be crucial here as well. However, we need to introduce just a little more notation before we can state our result. Let $n\in \Z$ and let $S$ be a finite subset of $\N$. Set

\begin{equation}
I_n(S):=\{(k_i)_{i\in S}: k_i \in \N \textrm{ and } \sum_{i\in S}k_i=n\},
\end{equation}
and define
\begin{equation}\label{hndef}
h_n^S:=\sum_{(k_i) \in I_n(S)} T^{\sum_{i\in S}k_i q^i} \in \A.
\end{equation}
Note that if $n<0$, then $I_n(S)=\emptyset$ and hence $h_n^S=0$. Also $h_0^S=1$. (These properties hold even if $S=\emptyset$ since $I_n(\emptyset)=\emptyset$ for all $n\neq 0$ and  $I_0(\emptyset)=\{\emptyset\}$).
We are now ready to state a Drinfeld multinomial Theorem.

\begin{theorem}[Multinomial Formula] \label{multinomialthm}
Let $\phi$ be a rank $r$ Drinfeld module over any $\A$-field $L$ given by
\begin{equation}\label{Ar}
\phi_T=T+\sum_{i=1}^r A_i \tau^i.
\end{equation}
For $m \in \N$ and $n\in \Z$ define the coefficients $c(n;m):=c(n;m;\phi)$ by
\begin{equation}\label{cnm}
\phi_{T^m}=\sum_{n=0}^{rm}c(n;m)\tau^n,
\end{equation}
and $c(n;m)=0$ for $n<0$ or $n>rm$.
Then for all $m,\,  n \geq 0$ we have
\begin{equation}\label{multinomial}
c(n;m)=\sum_{{\bf S}\in P_r(n)} {\bf A^S}\cdot h_{m-|{\bf S}|}^{(\cup {\bf S}\cup \{n\})}.
\end{equation}
\end{theorem}

\begin{proof}
We proceed by induction on $m$. It is easy to verify that formula \eqref{multinomial} gives $c(0;0)=1$ since $P_r(0)=\{\emptyset\}$. For $n>0$ and ${\bf S}\in P_r(n)$ we must have $|{\bf S}|>0$,
hence $m-|{\bf S}|<m$, and thus formula \eqref{multinomial} gives $c(n;0)=0$ for all $n>0$. So the statement is valid for $m=0$ since $\phi_1=1$.
Next we note that, because of the identity $\phi_{T^{m+1}}=\phi_T(\phi_{T^m})$,
the coefficients $c(n;m)$ satisfy the recursion formula
\begin{equation}\label{crecursion}
c(n;m+1)=Tc(n;m)+\sum_{i=1}^r A_i\cdot c(n-i;m)^{q^i}.
\end{equation}
Thus, by the induction hypothesis we have
\begin{equation}\label{step1}
c(n;m+1)=T\sum_{{\bf S}\in P_r(n)} {\bf A^S}\cdot h_{m-{\bf |S|}}^{(\cup {\bf S}\cup \{n\})}
+\sum_{i=1}^r A_i \sum_{{\bf S}^{(i)}\in P_r(n-i)} \left({\bf A}^{{\bf{S}}^{(i)}}\right)^{q^i} \left(h_{m- |\bf{S}^{(i)}|}^{(\cup {\bf S}^{(i)}\cup \{n-i\})}\right)^{q^i}
\end{equation}
Note that $A_j^{q^i\cdot w(S_j^{(i)})}=A_j^{w(S_j^{(i)}+i)}$,  $A_i\cdot A_i^{q^i\cdot w(S_i^{(i)})}=A_i^{w(\{0\}\cup (S_j^{(i)}+i))}$, and that
\[
 \left(h_{m- |{\bf S}^{(i)}|}^{(\cup {\bf S}^{(i)}\cup \{n-i\})}\right)^{q^i}= h_{m- |{\bf S}^{(i)}|}^{(\cup ({\bf S}^{(i)}+i)\cup \{n\})}.
\]
Using the identification \eqref{map} of $P_r(n-i)$ and $P_r^i (n)$ we see that \eqref{step1} becomes
\begin{equation}\label{step2}
\begin{split}
c(n;m+1)&=\sum_{i=1}^r \sum_{{\bf S} \in P_r^i(n)} {\bf A^S}
\cdot\left[Th_{m-|{\bf S}|}^{(\cup {\bf S}\cup \{n\})}+ h_{m-[(|S_i|-1)+\sum_{j\neq i} |S_j|]}^{(\cup_{j=1}^r (S_j\setminus \{0\})\cup\{n\})}\right]\\
&=\sum_{i=1}^r \sum_{{\bf S} \in P_r^i(n)} {\bf A^S}\cdot h_{m+1-|{\bf S}|}^{(\cup {\bf S}\cup \{n\})},
\end{split}
\end{equation}
which proves the result, since the summands are uniform in $i$ and the sets $P_r^i(n)$ partition $P_r(n)$.
\end{proof}

\begin{corollary} \label{supersingular}
Let $\frakp \in \A$ be a monic prime of the form
\begin{equation}
\frakp=\sum_{i=0}^{d} \mu_i T^i,
\end{equation}
and let $L_\frakp$ be a field extension of $\A/\frakp$. Let $\phi$ be a rank $2$ Drinfeld module over $L_\frakp$ given by
\[
\phi_T={T}+A\tau+B\tau^2,
\]
and let $c(n;m)$ be as in \eqref{multinomial}, then $\phi$ is supersingular at $\frakp$ if and only if
\begin{equation}\label{SS}
\sum_{i=\lceil\frac{d}{2}\rceil}^d \mu_ic(d;i)\equiv 0 \pmod \frakp .
\end{equation}
\end{corollary}

\begin{proof}
We have
\[
\phi_{\mu_0+\dots+\mu_dT^d}=\sum_{i=0}^d\mu_i\sum_{n=0}^{2i} c(n;i)\tau^n
=\sum_{n=0}^{2d}\left(\sum_{i=\lceil\frac{n}{2}\rceil}^d \mu_ic(n;i)\right)\tau^n,
\]
and the result follows by recognizing the coefficient of $\tau^d$.
\end{proof}

\begin{example}
To illustrate the results above, we identify the condition for a rank $2$ Drinfeld module to be supersingular at a degree $4$ monic prime $\frakp$. By \eqref{SS} this is equivalent to the vanishing modulo $ \frakp$  of
\begin{align}
\mu_2 c(4;2)&{}+\mu_3 c(4;3)+c(4;4) \label{ss4}\\
&=\mu_2B^{1+q^2}+\mu_3[B^{1+q^2}(T+T^{q^2}+T^{q^4})+A^{1+q}B^{q^2}+A^{1+q^3}B^q+A^{q^2+q^3}B] \notag \\
&\hspace*{10pt} {}+B^{1+q^2}(T^2+T^{2q^2}+T^{2q^4}+T^{1+q^2}+T^{1+q^4}+T^{q^2+q^4}) \notag \\
&\hspace*{10pt} {}+A^{1+q}B^{q^2}(T+T^q+T^{q^2}+T^{q^4})
+A^{1+q^3}B^{q}(T+T^q+T^{q^3}+T^{q^4}) \notag \\
&\hspace*{10pt} {}+A^{q^2+q^3}B(T+T^{q^2}+T^{q^3}+T^{q^4})+A^{1+q+q^2+q^3}. \notag
\end{align}
Note that $\{T^{q^i} \pmod \frakp, 0\leq i\leq 3\}$ are the $4$ distinct roots of $\frakp$ in $L_\frakp$. Thus we have the following congruences
\begin{equation}\label{modp}
\begin{split}
\mu_2&\equiv T^{1+q}+T^{1+q^2}+T^{1+q^3}+T^{q+q^2}+T^{q+q^3}+T^{q^2+q^3} \pmod \frakp, \\
\mu_3&\equiv -(T+T^q+T^{q^2}+T^{q^4}) \pmod \frakp,\\
T^{q^4}&\equiv T \pmod \frakp.
\end{split}
\end{equation}
Substituting \eqref{modp} into \eqref{ss4}, a simple computation yields
\begin{multline*}
\mu_2 c(4;2)+\mu_3 c(4;3)+c(4;4) \\ =A^{1+q+q^2+q^3}-[1]A^{q^2+q^3}B-[2]A^{1+q^3}B^q-[3]A^{1+q}B^{q^2}+[2][3]B^{1+q^2}.
\end{multline*}
Finally, dividing the above expression by $B^{1+q^2}$ we see that $\phi$ is supersingular at $\frakp$ if and only if $\jmath(\phi)$ is a root of
\begin{equation}\label{ssj4}
\jmath^{q^2+1}-[1]\jmath^{q^2}-[2]\jmath^{q^2-q+1}-[3]\jmath+[1][3] \equiv 0 \pmod \frakp.
\end{equation}
Using methods from Drinfeld modular forms, Cornelissen~\cite{Cornelissen99a}, \cite{Cornelissen99b} has also developed recursive formulas for polynomials defining supersingular $\jmath$-invariants, and one can successfully compare \eqref{ssj4} with $P_4(\jmath)$ in \cite[(2.2)]{Cornelissen99b}.
\end{example}


\begin{thebibliography}{999}

\bibitem{Carlitz35} L. Carlitz, \emph{On certain functions connected with polynomials in a Galois field}, Duke Math. J. \textbf{1} (1935), 137--168.

\bibitem{Carlitz95} L. Carlitz, \emph{Chapter 19 of ``The arithmetic of polynomials''}, Finite Fields Appl. \textbf{1} (1995), 157--164.

\bibitem{Cornelissen99a} G. Cornelissen, \emph{Deligne's congruence and supersingular reduction of Drinfeld modules}, Arch. Math. (Basel) \textbf{72} (1999), 346--353.

\bibitem{Cornelissen99b} G. Cornelissen, \emph{Zeros of Eisenstein series, quadratic class numbers and supersingularity for rational function fields}, Math. Ann. \textbf{314} (1999), 175--196.

\bibitem{GekelerBook} E.-U. Gekeler, \emph{Drinfeld modular curves}, Lecture Notes in Mathematics, vol. 1231, Springer-Verlag, Berlin, 1986.

\bibitem{Gekeler88} E.-U. Gekeler, \emph{On the coefficients of Drinfeld modular forms},
Invent. Math. \textbf{93} (1988),  667--700.

\bibitem{Gekeler91} E.-U. Gekeler, \emph{On finite Drinfeld modules},
J. Algebra \textbf{141} (1991), 187--203.

\bibitem{GossBook} D. Goss, \emph{Basic Structures of Function Field Arithmetic},
Springer, Berlin, 1996.

\bibitem{Hayes79} D.~R. Hayes, \emph{Explicit class field theory in global function fields}, in: Studies in algebra and number theory, Adv. in Math. Suppl. Stud. \textbf{6}, Academic Press, New York, 1979, pp.~173--217.

\bibitem{KedlayaBook} K.~S. Kedlaya, \emph{$p$-adic Differential Equations}, Cambridge Univ. Press, Cambridge, 2010.

\bibitem{LutesThesis} B.~A. Lutes, \emph{Special values of the Goss $L$-function and special polynomials}, Ph.D. thesis, Texas A\&M University, 2010.  (Available at http://www.math.tamu.edu/$\sim$map/.)

\bibitem{LutesP} B.~A. Lutes and M.~A. Papanikolas, \emph{Algebraic independence of values of Goss $L$-functions at $s=1$}, arXiv:1105.6341, 2011.


\bibitem{Thakur92} D.~S. Thakur, \emph{Drinfeld modules and arithmetic in the function fields}, Internat. Math. Res. Notices \textbf{1992} (1992), no. 9, 185--197.

\bibitem{Thakur93} D.~S. Thakur, \emph{Shtukas and Jacobi sums}, Invent. Math. \textbf{111} (1993), 557--570.

\bibitem{ThakurBook} D.~S. Thakur, \emph{Function Field Arithmetic}, World Scientific Publishing, River Edge, NJ, 2004.

\bibitem{Wade46} L.~I. Wade, \emph{Remarks on the Carlitz $\psi$-functions}, Duke Math. J. \textbf{13} (1946), 71--78.

\end{thebibliography}
\end{document}